\documentclass{amsart}%--Ken
\usepackage{amsmath}
\usepackage{amssymb} 
\usepackage{amscd} 
\usepackage{graphicx}   %\usepackage[dvipdfmx]{graphicx}
\usepackage{xcolor}
\usepackage{marvosym}
\usepackage{float}
\usepackage{ascmac}
\usepackage{wrapfig}
\usepackage{epsfig}
\usepackage{tikz}

\usepackage{hyperref}
\hypersetup{
	colorlinks=true,
	linktoc=all,
    	linkcolor={red!50!black},
   	citecolor={blue!50!black},
  	urlcolor={blue!80!black}
	} %%% This tones down the bright green hyperlink box and its other colors...

\allowdisplaybreaks % To make pagebreaks in equations

\tikzset{
  pt/.style={insert path={node[scale=1.5]{.}}},
  dnup/.style={insert path={ [pt] .. controls +(0,1em) and +(0,-1em) .. +(#1em,2em) [pt]}},
  dndn/.style={insert path={ [pt] .. controls +(0,1em) and +(0,1em) .. +(#1em,0) [pt]}},
  upup/.style={insert path={ [pt] .. controls +(0,-1em) and +(0,-1em) .. +(#1em,0) [pt]}},
}

\usepackage{rotating}

\usepackage{booktabs} %for nice tables

\usepackage{array}%for vertical centering in tables
\newcolumntype{M}[1]{>{\centering\arraybackslash}m{#1}} %for vertical centering in tables

\newcommand{\R}{{\mathbb R}}

\newcommand{\Z}{\mathbb Z}
\newcommand{\N}{\mathbb N}

\newcommand{\nbhd}{\mathcal{N}}
\newcommand{\bdry}{\partial}

\newtheorem{theorem}{Theorem}[section]
\newtheorem{lemma}[theorem]{Lemma}

\newtheorem{corollary}[theorem]{Corollary}
\newtheorem{claim}[theorem]{Claim}

\newtheorem*{SatCNC}{The Satellite Crossing Number Conjecture}

\newtheorem*{StabCNC}{The Stable Crossing Number Conjecture}
\newtheorem*{EnhancedStabCNC}{The Enhanced Stable Crossing Number Conjecture}
\newtheorem*{WeakStabCNC}{The Weak Stable Crossing Number Conjecture}
\newtheorem*{StabBIC}{The Stable Braid Index Conjecture}
\newtheorem*{WeakStabBIC}{The Weak Stable Braid Index Conjecture}
\newtheorem*{thm:SCN}{Theorem~\ref{thm:SCN}}
\newtheorem*{thm:stablecrossing}{Theorem~\ref{thm:stablecrossing}}
\newtheorem*{thm:satellite}{Theorem~\ref{thm:satellite}}
\newtheorem*{cor:coherent}{Corollary~\ref{cor:coherent}}
\newtheorem*{cor:noncoherent}{Corollary~\ref{cor:noncoherent}}
\newtheorem*{thm:stablebraid}{Theorem~\ref{thm:stablebraid}}
\newtheorem*{cor:boundsoncrossing}{Corollary~\ref{cor:boundsoncrossing}}

\theoremstyle{definition}
\newtheorem{definition}[theorem]{Definition}
\newtheorem{example}[theorem]{Example}

\newtheorem{remark}[theorem]{Remark}   %%%%  I moved Remarks here so that they wouldn't be in italics

\numberwithin{equation}{section}
\numberwithin{figure}{section}
\numberwithin{table}{section}

\definecolor{dartmouthgreen}{rgb}{0.05, 0.5, 0.06}

\renewcommand{\)}{\textup{)}}

\usepackage{fullpage}

\begin{document}
\baselineskip 14pt

\title{
The stable crossing number of a twist family of knots \\ and the satellite crossing number conjecture}

\author[K.L. Baker]{Kenneth L. Baker}
\address{Department of Mathematics, University of Miami, 
Coral Gables, FL 33146, USA}
\email{k.baker@math.miami.edu}

\author[K. Motegi]{Kimihiko Motegi}
\address{Department of Mathematics, Nihon University, 
3-25-40 Sakurajosui, Setagaya-ku, 
Tokyo 156--8550, Japan}
\email{motegi.kimihiko@nihon-u.ac.jp}

\begin{abstract}
Twisting a given knot $K$ about an unknotted circle $c$ a full $n \in \N$ times, 
we obtain a ``twist family'' of knots $\{ K_n \}$.  Work of Kouno-Motegi-Shibuya implies that for a non-trivial twist family
the crossing numbers $\{c(K_n)\}$ of the knots in a twist family grows unboundedly. However potentially this growth is  rather slow and may never become monotonic.  Nevertheless, based upon the apparent diagrams of a twist family of knots, one expects the growth should eventually be linear.
Indeed we conjecture that 
if $\eta$ is the geometric wrapping number of $K$ about $c$, then the crossing number of $K_n$ grows like $n \eta(\eta-1)$ as 
$n \to \infty$. 
To formulate this, we introduce the ``stable crossing number'' of a twist family of knots and establish the conjecture for 
(i) coherent twist families where the geometric wrapping and algebraic winding of $K$ about $c$ agree
and
(ii) twist families with wrapping number $2$ subject to an additional condition. 
Using the lower bound on a knot's crossing number in terms of its genus via Yamada's braiding algorithm, we bound the stable crossing number from below using the growth of the genera of knots in a twist family.  (This also prompts a discussion of the ``stable braid index''.) 
As an application, we prove that
highly twisted satellite knots in a twist family where the companion is twisted as well satisfy the Satellite Crossing Number Conjecture. 
\end{abstract}

\maketitle

 \renewcommand{\thefootnote}{}

\tableofcontents

%%%%%%%%%%%%%%%%%%%%%%%%%%%

\section{Introduction}
\label{intro}

Crossing number is the most elementary knot invariant. 
It measures an `appearance complexity' of knots with the following distinguished property that leads to its effective use in knot tabulation:
For any given integer $n > 0$ there are only finitely many knots $K$ whose crossing number is smaller than $n$.

Let $K$ be a knot in $S^3$.
We may assume $K$ is disjoint from two points $\pm\infty \in S^3$ so that $K \subset S^3 - \{\pm \infty\} = S^2 \times \mathbb{R}$.  Then for a generic projection $p \colon S^2 \times \mathbb{R} \to S^2$ 
we may assume that $p|_K \colon K \to S^2$ is one-to-one except for finitely many double points where 
$p(K)$ crosses itself once transversely. 
Assign ``over/under'' information at each double point according to the height of the two preimage points to obtain a \textit{knot diagram} $D$ of  $K$.
A double point with over/under information is called a \textit{crossing point} of $D$ and the number of crossing points of $D$ is denoted by $c(D)$. 
Then the \textit{crossing number} $c(K)$ of $K$ is defined as
\[
c(K) = \mathrm{min}\{ c(D) \mid D\ \textrm{is a diagram of}\ K \}.
\] 

In spite of the simplicity of its definition, 
crossing number is notoriously intractable in general.
For instance, although it seems very natural to expect, the following conjecture and its two strengthenings remains widely open.
\begin{SatCNC}
Let the knot $K$ be a non-trivial satellite of a knot $k$. 
\begin{itemize}
	\item Then $c(K) \geq c(k)$. (Problem 1.67 \cite{Kir})
	\item Furthermore, if $K$ is a winding number $\omega \geq 1$ satellite of $k$, then $c(K) \geq \omega^2 c(k)$ \cite{He}.
	\item More strongly, if $K$ is a wrapping number $\eta \ge 1$ satellite of $k$, then $c(K) \ge \eta^2 c(k)$.\)

\end{itemize}
\end{SatCNC}

Important progress was made by Lackenby \cite{Lac_sate}, 
who proves that $\displaystyle c(K) \ge \frac{1}{10^{13}}c(k)$.
For details and related results, see Lackenby \cite{Lac_sum} and \cite{FH, KL, Ito, BMT_Mazur}. 

\medskip

We are intrigued by the behavior of crossing numbers under twisting. 

Given a knot $K$ in $S^3$ and a disjoint unknot $c$, then for each integer $n$ doing $-1/n$ Dehn surgery on $c$ produces a knot $K_{c,n}$ in $S^3$.  
Effectively, $K_{c,n}$ is the result of twisting $K$ about $c$ a full $n$ times.  
Thus we have a \textit{twist family} of knots $\{ K_{c,n} \}$ indexed by the integers $n$ where $K=K_{c,0}$; see Figure~\ref{fig:twist_def}.  
With the twisting circle $c$ understood, we henceforth drop it from the notation and speak of the twist family $\{K_n\}$.

\begin{figure}[!ht]
\includegraphics[width=0.45\linewidth]{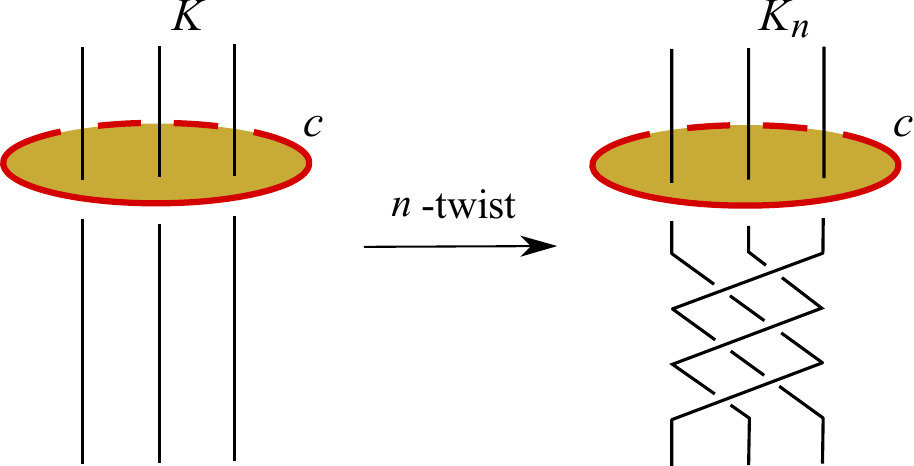}
\caption{Twisting knots and twisted region of a knot diagram}
\label{fig:twist_def}
\end{figure}

For a twist family of knots $\{K_n\}$ with twisting circle $c$, the \textit{wrapping number} $\eta$  and the \textit{winding number} $\omega$ of $K$ about $c$ measure the minimal geometric and algebraic intersection numbers of $K$ with a disk bounded by $c$.  
We choose orientations so that $\omega$ is non-negative.  
In Definition~\ref{defn:norm} we introduce the {\em meridional norm} $x$ of the twist family, the Thurston norm of the homology class dual to the meridian of $c$ in the exterior of $K\cup c$.
With Lemma~\ref{lem:meridionalnorm} it follows that 
\[0 \leq \omega \leq x+1 \leq \eta \qquad \mbox{and} \qquad \omega \equiv x+1 \equiv \eta \pmod{2}.\]   
Note that if $\eta = 0$ or $1$, then $K_{n}=K$ for all $n \in\Z$. Hence we regard twist families with $\eta = 0$ or $1$ as trivial.
Thus we henceforth implicitly assume $\eta \geq 2$ unless otherwise stated.

Observe that the crossing numbers of knots in a non-trivial twist family with must grow.
For a given integer $n$, 
there are only finitely many integers $m$ such that 
$K_m$ is isotopic to $K_n$ by \cite[Theorem~3.2]{KMS}. 
On the other hand, for a given integer $C$, 
there are only finitely many knots whose crossing number is less than or equal to $C$. 
Hence we obtain
\begin{lemma}
\label{lem:knot_types}
Let $\{K_n\}$ be a twist family with wrapping number $\eta \geq 2$.
For any given constant $C > 0$, 
there is a constant $N_C > 0$ such that $c(K_n) > C$ for all integers $n$ with $|n|  \ge N_C$. \qed
\end{lemma}

\begin{figure}[!ht]
\includegraphics[width=0.9\linewidth]{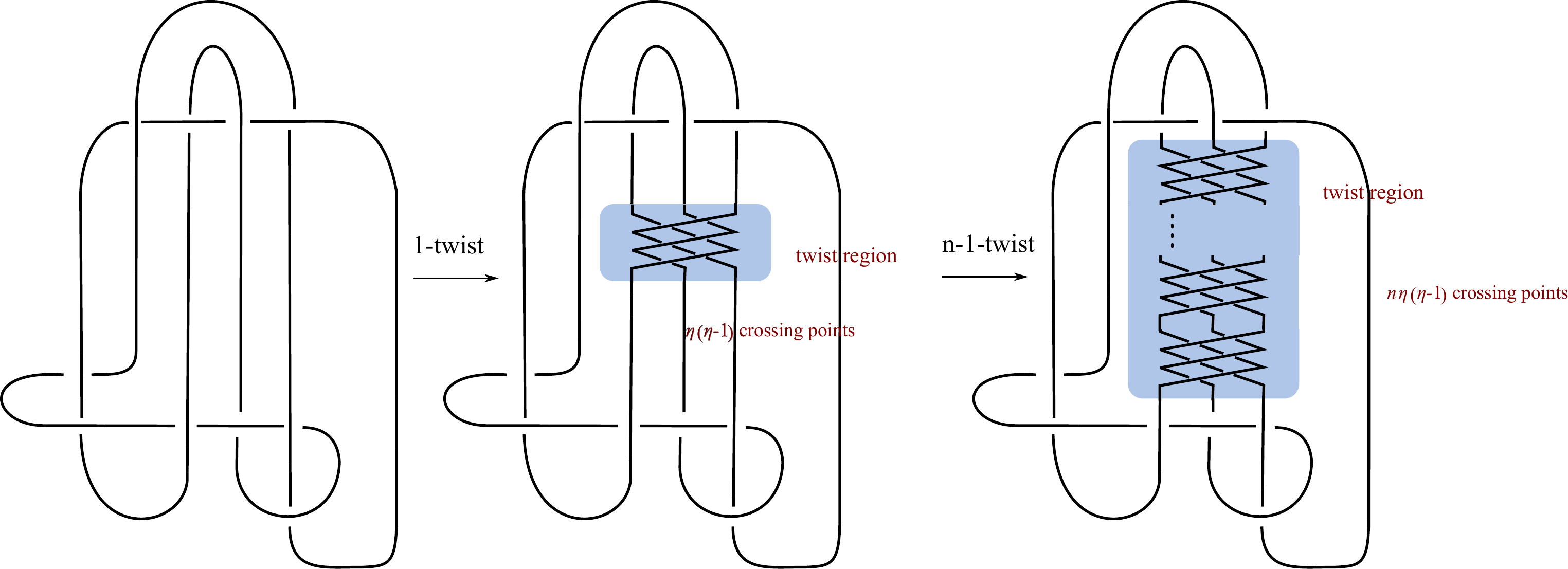}
\caption{The twist region of the diagram}
\label{fig:twisting_region}
\end{figure}

We would like to quantify this growth. 
A twist family $\{K_n\}$ has a diagram in which the twisting circle is planar and bounds a disk that the projection of $K$ crosses in $\eta$ parallel arcs as shown in Figure~\ref{fig:twist_def}.  
A diagram of $K_n$ for $n \geq 0$ is obtained by fully twisting these arcs $n$ times. 
Counting the crossing number arising in the twisted region (Figure~\ref{fig:twisting_region}) of the diagram of $K_n$, one observes that $\varlimsup_{n\to \infty} c(K_n) \leq c(K) + n\eta(\eta-1)$.
This prompts examining the asymptotic behavior of the sequence of crossing numbers ---or indeed of any invariant--- of the knots in a twist family.

\begin{definition}[Stable crossing number]
\label{defn:SCN}
Given a twist family of knots $\{K_n\}$, let $\{c(K_n)\}$ be the corresponding sequence of crossing numbers.
Define
\[\overline{c_s}(K_n) = \varlimsup_{n \to \infty} \frac{c(K_n)}{n} \qquad \mbox{and} \qquad \underline{c_s}(K_n) = \varliminf_{n \to \infty} \frac{c(K_n)}{n}.\]
Should these two be equal, define the {\em stable crossing number} of $\{K_n\}$ to be 
\[ c_s(K_n) = \lim_{n \to \infty} \frac{c(K_n)}{n}. \]
\end{definition}

In this notation, the previous observation becomes $\overline{c_s}(K_n) \leq \eta(\eta-1)$.  
We propose that this is not just an upper bound, but actually the stable crossing number.

\begin{StabCNC}
Let $\{ K_n \}$ be a twist family of knots with wrapping number $\eta$. 
Then the stable crossing number exists and is $c_s(K_{n}) = \eta(\eta-1)$.
\end{StabCNC}

In fact, we expect the following more precise conjecture. 

\begin{EnhancedStabCNC}
\label{enhanced}
 Let $\{ K_n \}$ be a twist family
of knots with wrapping number $\eta$. 
Then for sufficiently large $n$, $c(K_{n+1}) = c(K_{n}) + \eta(\eta-1)$. 
\end{EnhancedStabCNC}

A weaker version of the Stable Crossing Number Conjecture asserting linear growth should be more attainable.

\begin{WeakStabCNC}
Let $\{ K_n \}$ be a twist family of knots with wrapping number $\eta>1$. 
Then $\underline{c_s}(K_{n}) >0$.
\end{WeakStabCNC}

Note that Lemma~\ref{lem:knot_types} does not imply the Weak Stable Crossing Number Conjecture.  
For example, possibly $c(K_n) < \sqrt{n} \varepsilon$ for large enough $n$ and some $\varepsilon>0$ so that $\underline{c_s}(K_{n}) = 0$.

Our main result addresses the Stable Crossing Number Conjecture, 
confirming it for {\em coherent} twist families, those with winding equal to wrapping. 
(Birman-Menasco might say such families are type 0 or type 1 \cite{BiMe} while Nutt would call them NRS for `non-reverse string' \cite{Nutt}.  
Non-coherent twist families would be either type $k$ for $k\geq2$ or RS for `reverse string'.)

\begin{thm:stablecrossing}
Let $\{ K_n \}$ be a twist family of knots with winding number $\omega$, meridional norm $x$, and wrapping number $\eta$.  
Then we have the following.  
\begin{enumerate}
\item If $\omega = \eta$, then $c_s(K_n) = \eta(\eta-1)$.
\item If $\omega < \eta$, then 
$\omega(\omega+1) 
\le 
\omega x 
\le 
\underline{c_s}(K_n)
\le \overline{c_s}(K_n)
\le \eta(\eta-1)$.
\end{enumerate}
\end{thm:stablecrossing}

The proof of Theorem~\ref{thm:stablecrossing} shows a refinement of the stable behavior of crossing number for coherent twist families.

\begin{cor:boundsoncrossing}
Let $\{ K_n \}$ be a twist family of knots with winding number $\omega$ and wrapping number $\eta$. 
\begin{enumerate}
\item If $\omega=\eta$, then there exist constants $C,\ C'$ such that
\[
C' + n\eta(\eta-1) \le c(K_n) \le C +  n\eta(\eta-1)
\]
for all integers $n \geq 0$.
\item If $\omega<\eta$, then there exist constants $C,\ C'$ such that  
\[
C' + n\omega(\omega\ +\ 1) \le c(K_n) \le C +  n\eta(\eta-1)
\]
for all integers $n \geq 0$.
\end{enumerate}
\end{cor:boundsoncrossing}

If $K$ is a satellite of a non-trivial knot $k$ with wrapping and winding numbers $\eta$ and $\omega$, 
then there is a solid torus neighborhood $V$ of $k$ that contains $K$ in its interior. 
We may view $(V,K)$ as a {\em pattern} $P$ for the satellite operation so that $K=P(k)$ and $K$ has wrapping and winding numbers $\eta$ and $\omega$ in $V$.  
Letting $c$ be a meridian of $V$ (and hence a meridian of $k$), twisting $K$ about $c$ produces the twist family $\{K_n\}$ of satellite knots of $k$. 
One may also view this as forming a twist family of 
patterns $\{P_n\}$ so that $K_n =P_n(k)$.

Lemma~\ref{lem:knot_types} implies that 
Satellite Crossing Number Conjecture holds true for ``sufficiently twisted patterns''. 

\begin{theorem}
\label{thm:hightwistedcrossing}
For any knot $k$ and any pattern $P$ with wrapping and winding numbers $\eta$ and $\omega$ such that $\eta\geq 2$, 
there exists a constant $N$ so that for every integer $n \geq N$, the satellite knot $K_n = P_n(k)$ satisfies 
\[c(K_n) > \eta^2 c(k) \ge \omega^2 c(k).\]
In particular, if $\omega \neq 0$, then for every integer $n \geq N$
\[
\pushQED{\qed} 
c(K_n) > c(k). \qedhere
\popQED
\]
\end{theorem}

In the setting of Theorem~\ref{thm:hightwistedcrossing}  
twisting about $c$ does not affect the knot type of the companion knot $k$ since it is a meridian, and hence 
$c(k_n) = c(k)$ is a constant.  

\medskip

We may also apply Theorem~\ref{thm:stablecrossing} to address the Satellite Crossing Number Conjecture for certain twist families of satellite knots, 
in which both the satellite knot $K$ and its companion knot $k$ are twisted.

Let $K$ be a satellite of a knot $k$ in $S^3$ so that $K$ is contained in the interior of a solid torus neighborhood $V$ of $k$. 
Let $c$ be an unknot disjoint from $V$.  
Then twisting $K$ and $k$ simultaneously about $c$ yields the twist families $\{K_n\}$ and $\{k_n\}$ of a satellite knot $K_n$ with companion $k_n$ for each integer $n$.  
The Satellite Crossing Number Conjecture would imply that $c(K_n) > c(k_n)$ for all integers $n$.

Let $\eta_K$ and $\omega_K$ be the wrapping and winding numbers of $K$ about $c$, and
let $\eta_k$ and $\omega_k$ be the wrapping and winding numbers of $k$ about $c$. 
Viewing $K \subset V$ as a {\em pattern} $P$ for the satellite operation, let $\eta_P$ and $\omega_P$ be the wrapping and winding numbers of $K$ about the meridian of $V$. (Alternatively, $\eta_P$ is the minimal geometric intersection number of $K$ and a meridian disk of $V$ while $\omega_P$ is the algebraic intersection number of $K$ and a meridian disk of $V$.) One readily observes from a homological argument that $\omega_K = \omega_k \omega_P$; Claim~\ref{claim:wrap_satellite} shows that $\eta_K=\eta_k \eta_P$.

\begin{figure}[h]
\centering
\includegraphics[width=.4\textwidth]{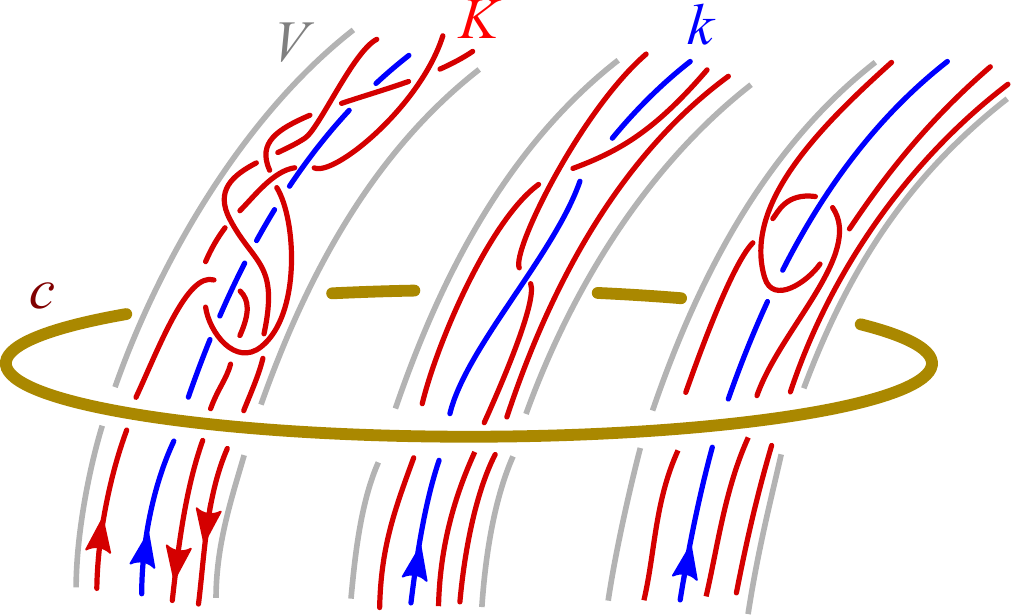}
\caption{Twisting satellite knots; $(\omega_k, \eta_k) = (3, 3)$ and $(\omega_P, \eta_P) = (1, 3)$.}
\label{twisting_satellite_2}
\end{figure}

\begin{thm:satellite}
With the above setup, suppose that  
$\omega_{P} \ge \eta_k/\omega_k$ and $\eta_{P} \ge 2$. 
Then there exists a constant $N > 0$ such that 
$c(K_n) > c(k_n)$ and 
$c(K_{-n}) > c(k_{-n})$ 
for all $n \ge N$. 

Furthermore, if $k$ is coherent and  the pattern has winding number $\omega_P \ge 2$, then there exists a constant $N' >0$ such that 
$c(K_n) > \omega_P^2 c(k_n)$ and $c(K_{-n}) > \omega_P^2 c(k_{-n})$ for all $n \geq N'$.
\end{thm:satellite}

\medskip

In particular, we have: 

\begin{corollary}
\label{coherent_satellite}
Assume that $k$ and the pattern $P$ are both coherent. 
Then we have a constant $N > 0$ such that $c(K_n) > \eta_P^2 c(k_n)$ and 
$c(K_{-n}) > \eta_P^2 c(k_{-n})$ 
for all $n \ge N$.  \qed
\end{corollary}

Our proof of Theorem~\ref{thm:satellite} basically compares the growth of crossing numbers $c(k_n)$ and $c(K_n)$ under twistings.  

There are two key ingredients to the proofs of our Theorems~\ref{thm:stablecrossing} and \ref{thm:satellite}.
The first is the lower bound on crossing number of a knot $K$ in terms of its knot genus $g(K)$ and braid index $b(K)$, observed by Diao \cite{Diao} (see also Ito \cite{Ito}) and derived from Yamada's approach \cite{Y} to Alexander's braiding theorem:
\[ 2g(K)-1 + b(K) \leq c(K)\]
We present this as Claim~\ref{claim:genus_braid_index_crossing}.
The second is the lower bound on knot genus in a twist family given in \cite[Theorem 5.3]{BT}.  
For a twist family of knots $\{K_n\}$ with twisting circle $c$, then for sufficiently large $n$ we have
\[2g(K_n) = G + n\omega x\]
where $G$ is some constant, $\omega$ is the winding number of $K$ about $c$, and $x$ is the meridional norm.
In other words, 
\begin{corollary}[of {\cite[Theorem 5.3]{BT}}]
The {\em stable genus} of a twist family of knots $\{K_n\}$ is $g_s(K_n) = \omega x/2$. \qed
\end{corollary}
Together these results establish the lower bound 
\begin{equation}
\label{eqn:lowerbound}
\omega x + \underline{b_s}(K_n) \leq \underline{c_s}(K_n).
\end{equation}
When $\omega=0$, any utility in this bound relies on information about the braid index. (Note that $x=0$ implies $\eta=1$.) 
Adapting Yamada's braiding algorithm \cite{Y}, we obtain

\begin{thm:stablebraid}
A twist family of knots $\{K_n\}$ with wrapping number $\eta$ and winding number $\omega$ satisfies 
\[ 0 \leq \underline{b_s}(K_n) \leq \overline{b_s}(K_n) \leq (\eta-\omega)/2.\]
\end{thm:stablebraid}

Unfortunately the general lower bound must be $0$ to accommodate coherent twist families; 
indeed Theorem~\ref{thm:stablebraid} shows that the stable braid index of a coherent twist family is $0$.  
Nevertheless we propose the following conjecture and its weaker version.  (Cf. \cite[Conjecture 3.12]{Nutt})
\begin{StabBIC}
The {\em stable braid index} of a twist family of knots $\{K_n\}$ with wrapping number $\eta$ and winding number $\omega$ is $b_s(K_n) = (\eta-\omega)/2$.
\end{StabBIC}

\begin{WeakStabBIC}
If a twist family of knots $\{K_n\}$ is non-coherent, then $\underline{b_s}(K_n) > 0$.
\end{WeakStabBIC}

\begin{remark}\phantom{xxxxxx}
\begin{enumerate}
\item Observe that if the Weak Stable Braid Index Conjecture were true, then the lower bound of Equation~(\ref{eqn:lowerbound}) would be positive when $\eta>1$.   
Hence Theorem~\ref{thm:hightwistedcrossing} 
would show the Satellite Crossing Number Conjecture is true for any satellite with $\eta >1$ and a sufficiently twisted pattern. 
\item Furthermore, Ohyama's bound $2b(K)-2 \leq c(K)$ \cite{Ohyama} gives the alternative lower bound 
\begin{equation}
\label{eqn:ohyamalowerbound}
2\underline{b_s}(K_n) \leq \underline{c_s}(K_n)
\end{equation}
 that would be sharper that Equation~(\ref{eqn:lowerbound}) whenever $\underline{b_s}(K_n)>\omega x$.
\end{enumerate}
\end{remark}

\medskip

As initial evidence towards the Stable Braid Index Conjecture, in Section~\ref{sec:wrap2wind0} we prove it using the Morton-Franks-Williams inequality \cite{Morton,FW} in the case of twist families of knots with wrapping number $2$ and winding number $0$ subject to mild constraints that preclude certain cancellations. 
A result of Ohyama \cite{Ohyama} then enables the confirmation of the Stable Crossing Number Conjecture for these twist families as well.

Finally in Section~\ref{sec:examples} we give some examples and questions.

\section{Meridional norm for the pair of a knot and a disjoint unknot}

For a pair of a knot $K$ and a disjoint unknot $c$, 
the wrapping number $\eta$ and the winding number $\omega$ are perhaps the most elementary invariants. 
Between them we introduce meridional norm as follows. 

\begin{definition}[\textbf{meridional norm}]
\label{defn:norm}
Let $K$ be a knot in $S^3$ and $c$ a disjoint unknot bounding a disk $\Delta$ with wrapping number $\eta$ and winding number $\omega$. 
Then the {\em meridional norm} of $K$ with respect to $c$, denoted by $x$, 
is the Thurston norm $x([D])$ of the homology class of the punctured disk $D = \Delta \cap (S^3 - \mathrm{int}N(K \cup c))$. 
\end{definition}

\begin{lemma}
\label{lem:meridionalnorm}
If $\omega = \eta$, then $x = \eta-1$.  Otherwise $\omega < \eta$ and $\omega+1 \le x \le \eta -1$. 
\end{lemma}

\begin{proof}
Let $D$ be a properly embedded punctured disk in the link exterior $E = S^3 - \mathrm{int}N(K \cup c)$ as in the definition of meridional norm so that $x=x([D])$.
Since the wrapping number $\eta$ is due to a disk bounded by $c$ that $K$ intersect $\eta$ times, we know
$x([D]) \le \eta-1$.   However, this punctured disk might not be norm minimizing.
Let $S$ be a properly embedded surface  in $E$ 
with coherently oriented boundary that both represents 
$[D]$ and realizes $x([D])$. 
Since $[S] = [D] \in H_2(E, \bdry E)$, 
$\bdry [S] = \bdry [D] \in H_1(\partial E)$ and 
$[\bdry S] = [\bdry D] = \omega \mu_K + \lambda_c$, 
where $\mu_K$ denotes a meridian of $K$ and $\lambda_c$ is a preferred longitude of $c$. 
The coherence of $\bdry S$ and that it is homologous to $\bdry D$ means that $\bdry S$ consists of $\omega$ curves in $\bdry N(K)$ representing 
meridian of $K$ and a single curve in $\bdry N(c)$ representing the preferred longitude of $c$. 
Since $S$ realizes $x([D])$ we may assume $S$ has no closed components (by discarding any sphere or torus components).
Then, if $S$ were not connected, the boundary of some component would be a coherently oriented non-empty collection of meridians of $K$; 
yet that component would cap off to a closed non-separating surface in $N(K)$. 
Hence we may take $S$ to be connected.  
Then one observes that
\[
x([D]) = - \chi (S) = 2g(S) - 2+ (\omega+1) = 2g(S) + \omega -1.
\]
Since $x([D]) \le \eta -1$, if $\eta = \omega$, 
then $g(S) = 0$. 
Conversely, 
if $g(S)=0$, then $S$ is a disk with $\omega$  coherently oriented punctures. 
Hence $\eta=\omega$.
Therefore, if $\omega < \eta$, then $g(S)\geq 1$ so that $x([D]) \geq \omega+1$.
Hence either $\omega=\eta$ and $\chi([D]) = \eta-1$ or $\omega < \eta$ and $\omega+1 \leq x([D]) \leq \eta-1$.
Indeed, if $\eta=\omega+2$, then $\chi([D])=\eta-1$.   
Note also that $x([D]) + 1 = 2g(S) + \omega \equiv \omega \equiv \eta  \pmod{2}$. 
\end{proof}

\begin{remark}
Assume that $\omega > 1$. 
Then Lemma~\ref{lem:meridionalnorm} shows that 
\[
1 < \frac{\omega +1}{\omega -1} \le \frac{x}{\omega -1} \le \frac{\eta -1}{\omega -1}.
\]
Hence, the ratio $\frac{x}{\omega-1}$ is $1$ exactly when $\eta=\omega$ and is at least $1+\frac{2}{\omega-1}$ when $\eta>\omega$.  
As we discuss in the paragraph preceding  \cite[Theorem 4.2]{BM}, this ratio can be any rational number $r \geq 1$.  
When $\omega=1$ and $\eta>1$, $x$ can be any positive even integer.  
When $\omega=0$, $x$ can be any positive odd integer.  
See also \cite[Theorem 1.1]{BM}.
\end{remark}

\begin{example}
\label{ex:meridional_norm}
Let $T$ be any $(3,3)$-tangle in $D^2 \times I$ such that any properly embedded disk $\Delta$ with $\bdry \Delta$ essential in the annulus $\bdry D^2 \times I$ intersects the tangle more than once.  
Let $K_T$ be a link obtained by using this tangle for $T$ in the left and $K'_T$ a link in the right of Figure~\ref{fig:norm_tangle}, and let $x$ and $x'$ be their meridional norms with respect to $c$.
Following Lemma~\ref{lem:meridionalnorm} we have $x=x'=2$. 
Concretely, in the left 
a disk punctured thrice by $K_T$ has coherent boundary orientation on $\bdry \nbhd(K_T)$ and realizes $x=2$, 
while in the right a similar disk punctured thrice by $K_T'$ has incoherent boundary and realizes $x'=2$.  
In this latter case we may also tube together two oppositely oriented boundary components of the punctured disk to form a  once-punctured torus punctured once $K_T'$ that is homologous to the disk punctured thrice and also realizes $x'=2$.

However, letting $K_{T,2}$ and $K'_{T,2}$ be $2$-strand cables of $K_T$ and $K'_T$, respectively, 
we have meridional norms  $x_2=5$ for $K_{T,2}$ but just $x'_2=3$ for $K'_{T,2}$. 
Indeed, for $K_{T,2}$, $c$ bounds a disk punctured six times by $K_{T,2}$ coherently which realizes $x_2=5$. 
On the other hand, for $K'_{T,2}$, the above once punctured torus for $K'_T$ is now punctured twice coherently by $K'_{T,2}$ and realizes $x_2'=3$.
\end{example}

\begin{figure}[!ht]
\includegraphics[width=0.4\linewidth]{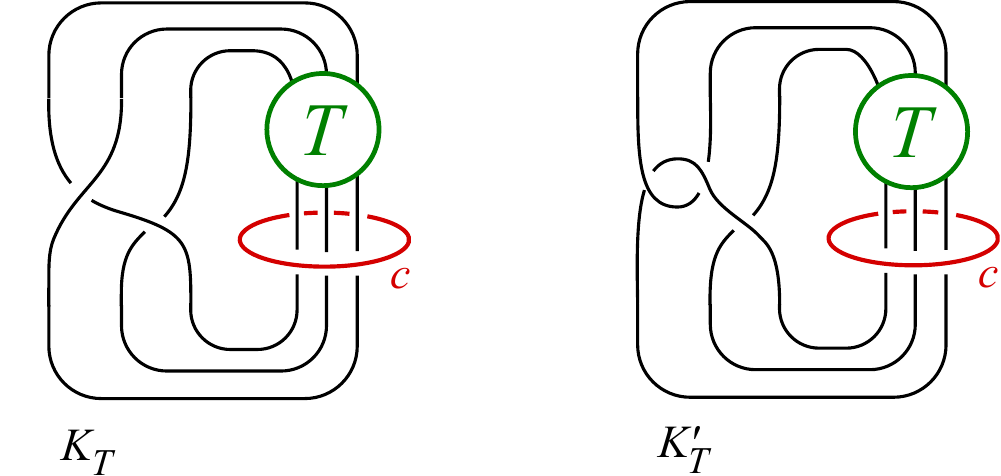}
\caption{
$K_T \cup c$ has $(\eta, \omega) = (3, 3)$, and $K'_T \cup c$ has $(\eta, \omega) = (3, 1)$, respectively, and 
both $K_T \cup c$ and $K'_T \cup c$ have the meridional norm $3$ unless $c$ cobounds an annulus with a meridian of a strand of $T$.
}
\label{fig:norm_tangle}
\end{figure}

\bigskip

\section{Proofs of Theorems~\ref{thm:stablecrossing} and \ref{thm:satellite}}
\label{proofs}

We first recall the following result of Diao \cite{Diao}, also demonstrated in \cite[Theorem 1.3]{Ito}, 
which essentially depends upon Yamada's refinement of Alexander's braiding theorem 
\cite[Theorem~3]{Y}. 

\begin{claim}
\label{claim:genus_braid_index_crossing}
For a knot $k$ in $S^3$, 
the following inequality holds. 
\[
2g(k) -1 + b(k) \le c(k).
\]
\end{claim}

\begin{proof}
Let us take a minimum crossing diagram $D$ of $k$, for which $c(D) = c(k)$. 
Applying Seifert's algorithm, we obtain a Seifert surface $S_D$ of $k$ which satisfies 
$\chi(S_D) = s(D) - c(D)$ where $s(D)$ denotes the number of Seifert circles.  
Note that $\chi(S_D) \le 1-2g(k)$. 
Now let $s(k)$ be the minimum number of Seifert circles among all diagrams of $k$. 
Then by \cite[Theorem~3]{Y}  we have $b(k) = s(k) \le s(D)$. 
Hence we get 
\[
1-2g(k) \ge \chi(S_D) = s(D) - c(D) \ge b(k) -c(k)
\]
from which the desired inequality follows.
\end{proof}

\subsection{Bounds on stable crossing numbers in twist families}

\begin{theorem}
\label{thm:stablecrossing}
Let $\{ K_n \}$ be a twist family of knots with winding number $\omega$, meridional norm $x$, and wrapping number $\eta$.  
Then we have the following.  
\begin{enumerate}
\item If $\omega = \eta$, then $c_s(K_n) = \eta(\eta-1)$.
\item If $\omega < \eta$, then 
$\omega(\omega+1) 
\le 
\omega x 
\le 
\underline{c_s}(K_n)
\le \overline{c_s}(K_n)
\le \eta(\eta-1)$.
\end{enumerate}
\end{theorem}

\begin{proof}
It follows from \cite{BT} that for sufficiently large $n$, 
\begin{equation}
\label{eqn:genus_norm}
2g(K_n) = G + n \omega x, 
\end{equation}
where $G$ is a constant depending only on the initial link $K \cup c$.  
With Claim~\ref{claim:genus_braid_index_crossing}, Equation (\ref{eqn:genus_norm}) shows that for sufficiently large $n$ we have
\begin{equation}
\label{eqn:norm_crossing}
G + n \omega x -1 + b(K_n)  = 2g(K_n) -1 + b(K_n) \le c(K_n).
\end{equation}
This implies that 
\begin{equation}
\label{eqn:wx_leq_scn}
\omega x \leq \omega x + \underline{b_s}(K_n) \leq  \varliminf_{n \to \infty}\frac{G + n \omega x - 1+ b(K_n)}{n} 
\le 
 \varliminf_{n \to \infty} \frac{c(K_n)}{n} = \underline{c_s}(K_n)
\end{equation}
where $\underline{b_s}(K_n) = \varliminf_{n\to \infty} b(K_n)/n  \geq 0$ since $b(K_n) >0$ for all $n$. 
(Corollary~\ref{cor:coherentbraid} shows that coherent twist families have $b_s(K_n)=0$.)
On the other hand, as in the discussion preceding Definition~\ref{defn:SCN},
we have a diagram $D_n$ of $K_n$ obtained from a `nice' diagram $D_0$ of $K=K_0$ which satisfies
\begin{equation}
\label{eqn:crossing_diagram}
c(K_n) \le c(D_n) = C + n\eta(\eta-1). 
\end{equation}
where the constant $C = c(D_0)$ is independent of $n$. 
Then 
\begin{equation}
\label{eqn:scn_leq_eta(eta-1)}
\overline{c_s}(K_n) = \varlimsup_{n \to \infty}\frac{c(K_n)}{n} \le 
\lim_{n \to \infty}\frac{c(D_n)}{n} 
= 
\lim_{n \to \infty}\frac{C + n\eta(\eta-1)}{n} 
= \eta(\eta-1).
\end{equation}

Together, Equations~(\ref{eqn:wx_leq_scn}) and (\ref{eqn:scn_leq_eta(eta-1)}) yield
\begin{equation}
\label{eqn:crossing_norm}
\omega x 
\le \omega x + \underline{b_s}(K_n) 
\le \underline{c_s}(K_n) 
\le \overline{c_s}(K_n)
\le \eta(\eta-1).
\end{equation}
Utilizing Lemma~\ref{lem:meridionalnorm}, we then obtain
the following from Equation~\ref{eqn:crossing_norm}:
\begin{itemize}
\item  If $\omega = \eta$ then
\[
\eta(\eta-1) 
= \omega x 
\le \underline{c_s}(K_n)
\le \overline{c_s}(K_n)
\le \eta(\eta-1)
\]
so that in fact $c_s(K) = \eta(\eta-1)$.
\item Otherwise $\omega < \eta$ and 
\[
\omega(\omega+1) 
\le \omega(\omega+1) +\underline{b_s}(K_n) 
\le \omega x +\underline{b_s}(K_n) 
\le \underline{c_s}(K_n)
\le \overline{c_s}(K_n)
\le \eta(\eta-1).
\]
\end{itemize}

\medskip

Now let $\overline{K \cup c} = \overline{K} \cup \overline{c}$ be the mirror image of $K \cup c$ and 
$\overline{K}_n$ be the knot obtained from $\overline{K}$ by $n$-twists along $\overline{c}$. 
The wrapping and winding numbers of $\overline{K}$ about $\overline{c}$ are also $\eta$ 
and  $\omega$. 
Since the mirror image  of $K_{-n}$ is $\overline{K_{-n}}$ which is $\overline{K}_{n}$ and the crossing number is preserved under mirroring, we have $c(K_{-n}) = c(\overline{K_{-n}}) = c(\overline{K}_{n})$.
Therefore, considering the stability under negative twisting, 
we again have $c_s(K) = \eta(\eta-1)$  if  $\omega=\eta$ and otherwise
\[
\omega(\omega+1) 
\le \omega x
\le \underline{c_s}(K_{-n})
\le \overline{c_s}(K_{-n})
\le \eta(\eta-1).
\]
\end{proof}

\begin{corollary}
\label{cor:boundsoncrossing}
Let $\{ K_n \}$ be a twist family of knots with winding number $\omega$ and wrapping number $\eta$. 
\begin{enumerate}
\item If $\omega=\eta$, then there exist constants $C,\ C'$ such that
\[
C' + n\eta(\eta-1) \le c(K_n) \le C +  n\eta(\eta-1)
\]
for all integers $n \geq 0$.
\item If $\omega<\eta$, then there exist constants $C,\ C'$ such that  
\[
C' + n\omega(\omega + 1) \le c(K_n) \le C +  n\eta(\eta-1)
\]
for all integers $n \geq 0$.
\end{enumerate}
\end{corollary}

\begin{proof}
Combining Equations
(\ref{eqn:norm_crossing}) and (\ref{eqn:crossing_diagram}) from the proof of Theorem~\ref{thm:stablecrossing}, we have 
\[
C' + n\omega x\le c(K_n) \le C +  n\eta(\eta-1)
\]
for all integers $n \geq 0$.
Then apply Lemma~\ref{lem:meridionalnorm} for the two cases of $\eta=\omega$ and $\eta>\omega$.
\end{proof}

\medskip

\begin{example}
\label{ex:Thurston_norm}
When the twist family is not coherent, using the meridional norm $x$ instead of the lower bound $\omega+1$ from Lemma~\ref{lem:meridionalnorm} can give improved estimates.

\begin{figure}[h]
	\begin{center}
		\includegraphics[width=0.55\textwidth]{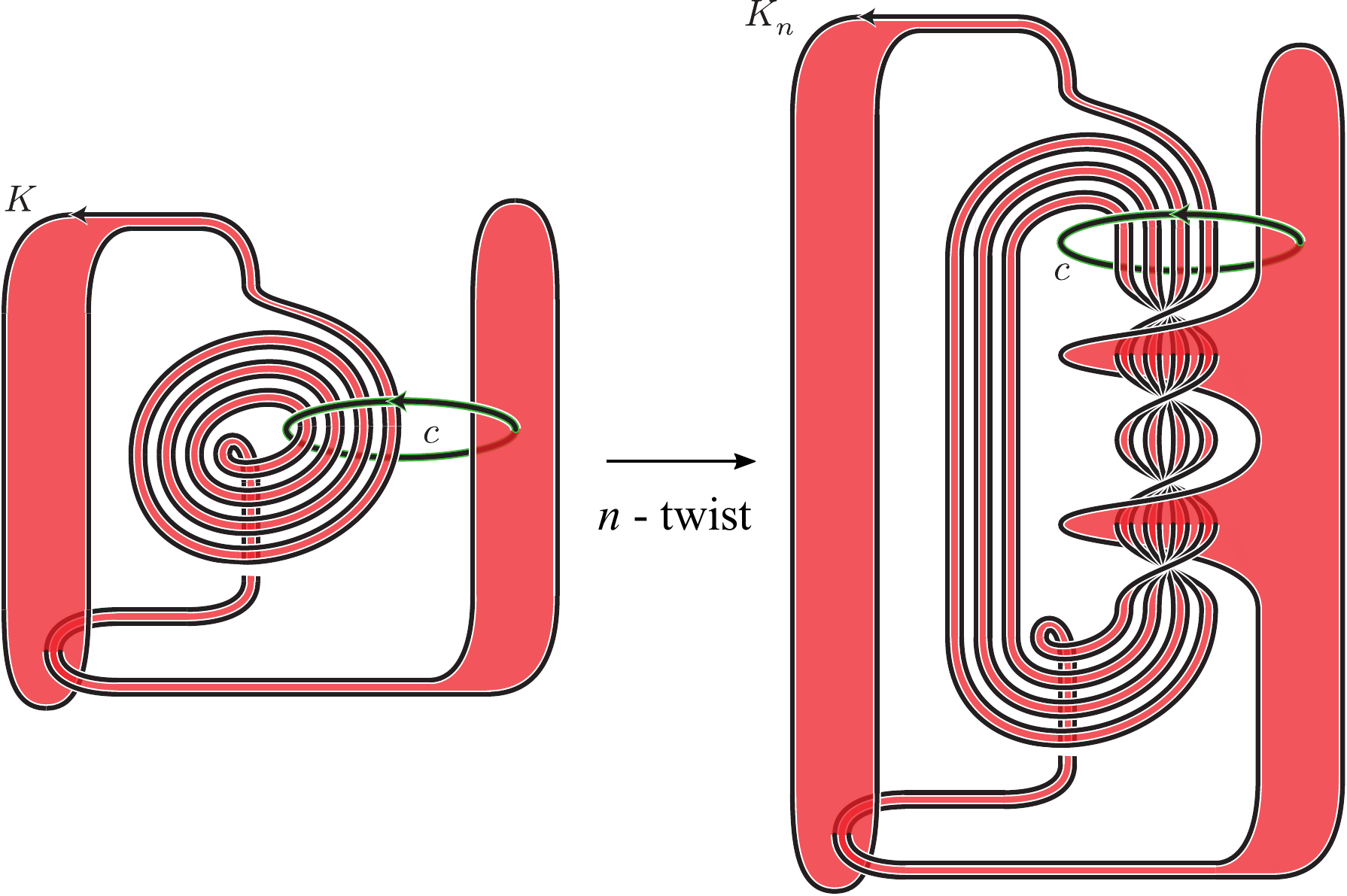}
		\caption{}
		\label{fig:wrappingribbons}
	\end{center}
\end{figure}

Take a link $K \cup c$ as in Figure~\ref{fig:wrappingribbons} in which $\eta = 9,\ \omega = 1$. 
As shown in \cite{BM}, 
$g(K_n) = 4n+1 = \dfrac{C + n x}{2}$ for sufficiently large $n$, 
and thus $\displaystyle \lim_{n \to \infty}\frac{g(K_n)}{n} = 4 = \dfrac{x}{2}$. 
Hence $x = 8$ and Theorem~\ref{thm:stablecrossing} shows
\[ 8 \le \underline{c_s}(K_n) \le \overline{c_s}(K_n) \le 72. \]
Note that since $\omega = 1< \eta$,  the lower bound using only the winding number is just $\omega(\omega+1) = 2$. 
\end{example}

\subsection{Crossing numbers of highly twisted satellites}
As in the introduction, let $K$ be a satellite of a knot $k$ in $S^3$ so that $K$ is contained in the interior of a solid torus neighborhood $V$ of $k$. 
Let $c$ be an unknot disjoint from $V$ so that twisting $K$ and $k$ simultaneously about $c$ yields the twist families $\{K_n\}$ and $\{k_n\}$ of a satellite knot $K_n$ with companion $k_n$ for each integer $n$. 

Let $\eta_K$ and $\omega_K$ be the wrapping and winding numbers of $K$ about $c$, and
let $\eta_k$ and $\omega_k$ be the wrapping and winding numbers of $k$ about $c$. 
Viewing $K \subset V$ as a {\em pattern} $P$ for the satellite operation, let $\eta_P$ and $\omega_P$ be the wrapping and winding numbers of $K$ about the meridian of $V$. We further assume $\eta_{P} \ge 2$.

\begin{claim}
\label{claim:wrap_satellite}
The winding number of $K$ about $c$ is $\omega_k\omega_P$, 
and the wrapping number of $K$ about $c$ is $\eta_k \eta_P$.
\end{claim}

\begin{proof}
Homological computations show that the winding number of $K$ about $c$ is the product $\omega_k\omega_P$. 
Now we show that the wrapping number of $K$ about $c$ is $\eta_k\eta_P$. 

Let $\eta_K$ be the wrapping number of $K$ about $c$.  
Take an embedded disk $D$ with $\bdry D = c$ for which $|D \cap k| = \eta_k$, 
the wrapping number of $k$ about $c$. 
Let $V$ be a solid torus neighborhood of $k$ that intersects $D$ in $\eta_k$ meridional disks. 
Viewing $K$ as a satellite of $k$, we may regard $K$ as being contained in the interior of $V$. 
Since the wrapping number of $K$ in $V$ is $\eta_P$, there is a meridional disk $D_V$ of $V$ so that $|D_V \cap K| = \eta_P$.  
Take $\eta_k$ parallel copies of $D_V$ (each that $K$ intersects $\eta_P$ times) and replace each of the $\eta_k$ disks of $D \cap V$ with a copy of $D_V$ and an annulus to form a new disk $D'$ that $K$ now intersects $\eta_k\eta_P$ times.  
By an appropriate choice of these annuli and a small isotopy in a collar neighborhood of $\bdry V$, $D'$ may be taken to be an embedded disk with boundary $c$.  
Hence $\eta_K \leq \eta_k \eta_P$.

Now we must show that $\eta_K \geq \eta_k \eta_P$. 
As $K$ is a satellite of $k$, there is an embedded torus $T$ that bounds a solid torus $V$ which is disjoint from $c$, contains $K$, and has $k$ as a longitude.  
Furthermore $V$ has a meridional disk $D_V$ that $K$ intersects $\eta_P$ times.
Among embedded disks bounded by $c$ that $K$ intersects $\eta_K$ times, let $D$ be one for which $|D \cap T|$ is minimal.
Then $D \cap T$ is a collection of circles.  
Among such circles that bounds a disk in $T$, if there is one, take $\gamma$ to be an innermost one.  
Replace the disk $\gamma$ bounds in $D$ with the disk it bounds in $T$ to make a new disk $D'$.  
A slight isotopy of $D'$ in a collar of the disk $\gamma$ bounds in $T$ pulls $D'$ off $T$ there so that $|D' \cap T| < |D \cap T|$.  
Since $T$ is disjoint from $K$, this also shows $|D' \cap K| \leq |D \cap K|$.  
Yet this either contradicts that the disk $D$ realizes the wrapping number $\eta_K$ or that it was chosen among such disks to minimize its intersections with $T$.  
Hence every circle of $D \cap T$ is essential in $T$.

Since $D \cap T$ is non-empty, a circle of $D \cap T$ that is innermost in $D$ bounds a disk $\delta$ in $D$.  
Since $\bdry \delta \subset D \cap T$, it must be an essential curve in $T$.  
Thus $\delta$ is a meridional disk of $V$.  
As such, we must have $|\delta \cap K| \geq \eta_P$.  
Moreover, this also implies that $D \cap T$ is a collection of meridians of $V$.  
Viewing $k$ as a longitude of $V$ in $T$, it may be isotoped in $T$ to intersect each of these meridians just once.  
Hence $|D \cap T| = |D \cap k| \geq \eta_k$.  
If every curve of $D \cap T$ is innermost in $D$, then $V$ intersects $D$ in a collection of meridional disks.  
Since $K$ intersects each of these meridional disks of $V$ at least $\eta_P$ times, it would then follow that $\eta_K \geq \eta_k \eta_P$.  

So suppose that some curve of $D \cap T$ is not innermost in $D$.  
Let $P$ be the component of $D - V$ that contains $c = \bdry D$ as depicted in Figure~\ref{fig:diskandT}. 
\begin{figure}[h]
	\begin{center}
		\includegraphics[width=0.17\textwidth]{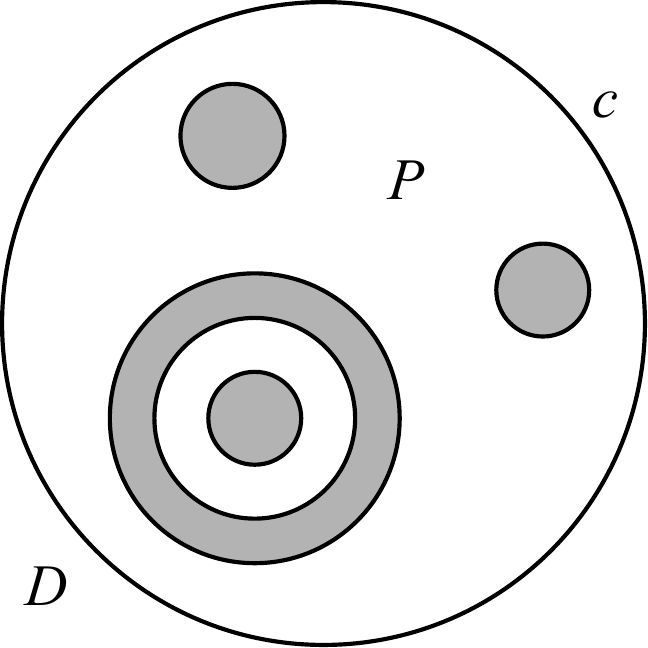}
		\caption{$D \cap T$ in $D$; shaded regions are inside of $V$.}
		\label{fig:diskandT}
	\end{center}
\end{figure}
Because $T = \partial V$ separates $c=\bdry D$ from $K$, $P \cap K = \emptyset$.  
The components of $\bdry P -c$ are components of $D \cap T$ and bound mutually disjoint disks in $D$.  
If one of these components is innermost in $D$, the subdisk it bounds is a meridional disk of $V$ that $K$ intersects at least $\eta_P$ times.  
If it is not innermost, then within the subdisk it bounds is at least one other component of $D \cap T$ that is innermost, and so the subdisk it bounds also intersects $K$ at least $\eta_P$ times.  
Hence this shows that $|D \cap K| \geq (|\bdry P|-1)\eta_P$ and $|D \cap T| > |\bdry P| -1$.

Now since every other component of $\bdry P$ besides $c$ belongs to $D \cap T$, each is a meridian of $T$.  
Thus each of these curves of $\bdry P - c$ can be connected by an annulus in $T$ to a parallel copy of $D_V$ to form a new disk $D'$.   
By an appropriate choice of these annuli and a small isotopy in a collar neighborhood of $T=\bdry V$, $D'$ may be taken to be an embedded disk with boundary $c$ that intersects $V$ in these $|\bdry P|-1$ copies of $D_V$.  
Hence $|D' \cap K| = (|\bdry P|-1)\eta_P$ and $|D' \cap T| = |\bdry P|-1$. But this contradicts our choice of $D$.
Thus every circle of $D \cap T$ is innermost in $D$, as desired.
\end{proof}

\begin{theorem}
\label{thm:satellite}
With the above setup, suppose that $\omega_{P} \ge \eta_k/\omega_k$, and $\eta_{P} \ge 2$.
Then there exists a constant $N > 0$ such that 
$c(K_n) > c(k_n)$ and 
$c(K_{-n}) > c(k_{-n})$ 
for all $n \ge N$. 

Furthermore, if $k$ is coherent and  the pattern has winding number $\omega_P \ge 2$, then there exists a constant $N' >0$ such that 
$c(K_n) > \omega_P^2 c(k_n)$ and $c(K_{-n}) > \omega_P^2 c(k_{-n})$ for all $n \geq N'$.
\end{theorem}

\begin{proof}
First we note that the assumption $\omega_{P} \ge \eta_k/\omega_k$ implies that $\omega_k \ne 0$, hence $\omega_k \ge 1$ and $\omega_P \ge 1$ (because $\eta_k \ge \omega_k \ge 1$). 
Also $\omega_K = \omega_k \omega_P \ge 1$.

The stable crossing number bounds 
\[
\overline{c_s}(k_n) \le \eta_k(\eta_k-1)\quad \textrm{and}\quad \underline{c_s}(K_n) \ge \omega_K(\omega_K + 1)\]
of Theorem~\ref{thm:stablecrossing} arise in its proof from the inequalities
\[ c(k_n) \le c(D_{k_0}) + n \cdot \eta_k(\eta_k-1) \quad \textrm{and}\quad  c(K_n) \ge G + n \cdot \omega_K x_K\]
where the first holds for any `nice' diagram $D_{k_0}$ of $k_0$ and all $n \geq 0$ while the second holds for some constants $G$ and $N_0$ (depending on the link $K_0 \cup c$) and for all $n \ge N_0$. (Here $x_K$ is the meridional norm of $K$ with respect to $c$.)
 Hence the difference
 \[ c(K_n) - c(k_n) \ge G-c(D_{K_0}) + n \cdot \left( \omega_K x_K - \eta_k(\eta_k-1) \right)\]
is positive for all $n \geq N_+$ for some $N_+ \geq N_0$ if and only if
\begin{equation}
\label{eqn:windingsatellite_vs_wrappingpattern}
\omega_K x_K > \eta_k(\eta_k-1).
\end{equation}
If $\omega_k=\eta_k$ and $\omega_P = \eta_P$ so that $\omega_K = \omega_k \omega_P = \eta_k \eta_P =  \eta_K$ by Claim~\ref{claim:wrap_satellite} and $x_K = \eta_K-1$ by Lemma~\ref{lem:meridionalnorm}, 
then 
\[
\omega_K x_K = \eta_K (\eta_K -1) = \eta_k\eta_P(\eta_k \eta_P -1).
\]
Since $\eta_P \ge 2$, 
we have $\eta_k\eta_P(\eta_k \eta_P -1) > \eta_k(\eta_k -1)$, 
which implies  Inequality~(\ref{eqn:windingsatellite_vs_wrappingpattern}). 

Otherwise, 
$\omega_k < \eta_k$ or $\omega_P < \eta_P$. 
Then by Claim~\ref{claim:wrap_satellite}, $\omega_K = \omega_k\omega_P < \eta_k\eta_P = \eta_K$, 
and Lemma~\ref{lem:meridionalnorm} shows that $x_K \geq \omega_K+1$.  
Hence $\omega_K x_K \geq \omega_K(\omega_K+1) > \omega_K(\omega_K -1)$ and then Inequality~(\ref{eqn:windingsatellite_vs_wrappingpattern}) holds if $\omega_K \geq \eta_k$.

Since $\omega_K = \omega_k \omega_P$, 
the assumption $\omega_P \geq \eta_k/\omega_k$ shows that 
$\omega_K = \omega_k \omega_P \ge \omega_k (\eta_k / \omega_k) = \eta_k$. 

It now follows that 
\[c(K_n) > c(k_n)\]
for all $n \geq N_+$ whenever $\omega_P \geq \eta_k/\omega_k$.

\medskip

Similarly,
the difference
 \[ c(K_n) - \omega_P^2 c(k_n) \ge G-\omega_P^2 c(D_{K_0}) + n \cdot \left( \omega_K x_K - \omega_P^2 \eta_k(\eta_k-1) \right)\]
is positive for all $n \geq N'_+$ for some $N'_+ \geq N_0$ if and only if
\[\omega_K x_K > \omega_P^2 \eta_k (\eta_k-1).\]
Thus, if $\omega_k=\eta_k$, 
then $x_K \geq \omega_K + 1 = \omega_k \omega_P +  1 = \eta_k\omega_P + 1$ so that 
\[
\omega_K x_K 
\ge \omega_k \omega_P(\eta_k \omega_P + 1) 
= \eta_k \omega_P(\eta_k \omega_P + 1)
> \eta_k^2\omega_P^2 - \eta_k\omega_P \geq \eta_k^2 \omega_P^2 - \eta_k \omega_P^2 = \omega_P^2\eta_k(\eta_k-1)
\]

as long as $\omega_P \geq 1$.
Then the above inequality holds, and we have that 
\[c(K_n) > \omega_P^2 c(k_n)\]
for all $n \geq N'_+$ whenever $\omega_k=\eta_k$ and $\omega_P \geq 1$.

\medskip

Analogous arguments show that we have a constants $N_-$ such that $c(k_{-n}) < c(K_{-n})$ 
for all $n \ge N_{-}$
and $N'_-$ such that $c(k_{-n}) < \omega_0^2 c(K_{-n})$ 
for all $n \ge N'_{-}$.
So now let $N$ be $\mathrm{max}\{ N_{+},\ N_{-}\}$ and $N'$ be $\mathrm{max}\{ N'_{+},\ N'_{-}\}$.
Then we have both that
\[c(K_n) > c(k_n) \quad \mbox{and} \quad c(K_{-n}) > c(k_{-n})\]
for all $n \geq N$ whenever $\omega_k \geq 1$ and $\omega_P \geq \eta_k/\omega_k$
and that
\[c(K_n) > \omega_P^2 c(k_n) \quad \mbox{and} \quad c(K_{-n}) \geq \omega_P^2 c(k_{-n})\]
for all $n \geq N'$ whenever $\omega_k=\eta_k$ and $\omega_P \geq 1$.
\end{proof}

\section{Stable braid index for a twist family of knots}
\label{sec:stable_braid_index}

\begin{theorem}
\label{thm:stablebraid}
Let $\{ K_n \}$ be a twist family of knots with wrapping number $\eta$ and winding number $\omega$. 
Then we have the following.  
\[ 0 \leq \underline{b_s}(K_n) \leq \overline{b_s}(K_n) \leq (\eta-\omega)/2.\]
\end{theorem}

Theorem~\ref{thm:stablebraid} immediately shows that the Stable Braid Index Conjecture holds true for coherent twist families.  

\begin{corollary}
\label{cor:coherentbraid}
Let $\{ K_n \}$ be a coherent twist family of knots. 
Then $b_s(K_n) = 0$. \qed
\end{corollary}

\begin{proof}[Proof of Theorem~\ref{thm:stablebraid}]
Since $b(K_n) \geq 1$, we trivially have the lower bound $0 \le  \underline{b_s}(K_n)$.
To show the upper bound, we observe that $K$ has a diagram as a closed braid in which the twisting circle and braid axis are conveniently arranged.

\begin{claim}
\label{claim:braidtwist}
Let $c$ be the twisting circle for $K$.
Then there exists a braid axis $A$ for $K$ that intersects $c$ in two points and a $2$--sphere containing $A \cup c$  so that
\begin{itemize}
\item each hemisphere bounded by $A$ is met by $K$ coherently and 
\item one hemisphere bounded by $c$ is split by an arc of $A$ into a disk that $K$ intersects $(\eta+\omega)/2$ times and another disk that $K$ intersects $(\eta-\omega)/2$ times.
\end{itemize}
In particular there is a diagram of $K \cup c$ as in Figure~\ref{fig:braidtwist}, 
in which the braid axis $A$ and the $2$--sphere is omitted. 
\end{claim}

\begin{proof}
Our twisting circle $c$ bounds a disk $\Delta$ that $K$ intersects $\eta$ times, 
$\frac{\eta + \omega}{2}$ times in one direction and $\frac{\eta-\omega}{2}$ times in the other.  
After an isotopy, $K$ and $\Delta$ may be assumed to appear as in Figure~\ref{fig:initialdiagram}(left) where a properly embedded arc $\alpha$ in $\Delta$ separates the intersections of $K$ of opposite directions.  

\begin{figure}[h]
\centering
\includegraphics[width=.18\textwidth]{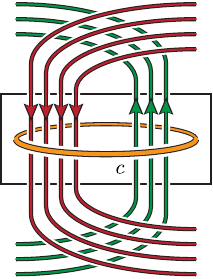}
\caption{A local picture of an initial diagram for $K \cup c$, shown for $\eta=7$ and $\omega=1$.}
\label{fig:initialdiagram}
\end{figure}

Working in $\R^3$, after a slight tilt and isotopy we then obtain a diagram in $\R^2$ of $K$ in which $\Delta$ projects to a line segment $\delta$ that $K$ crosses $\frac{\eta + \omega}{2}$ times in one direction and $\frac{\eta-\omega}{2}$ times in the other direction and $\alpha$ projects to a point $a$ in $\delta$ that separates these two sets of intersections.  
Furthermore we may arrange that in a neighborhood of $\delta$ there are no other arcs of $K$. 
Extend $\delta$ to a line, let $S_\delta$ be the preimage of this line in $\R^3$, and observe that $\Delta \subset S_\delta$.  
Let $A \subset S_\delta$ be the preimage of $a$ in $\R^3$.   
We will braid $K$ about $A$ by an isotopy with support outside of a neighborhood of $\Delta$.

View $K$ as an oriented polygonal knot in $\R^3$ in which the intersections with $S_\delta$ are all interior points of edges of $K$ that do not cross other edges of $K$ in the projection to $\R^2$. 
The edges that meet $\Delta$ all wind in the same direction around $A$.  
We may now apply Alexander's braiding algorithm \cite{Alex} almost verbatim, ensuring that any triangle move is disjoint from $\Delta$.  
In essence, if an edge $e$ of $K$ goes in the opposite direction, we find a solid triangle $T$ with $e$ as one edge so 
that $A$ intersects the interior of $T$ while 
$T-e$ is disjoint from $K \cup \Delta$, and then we replace $e$ with the other two edges of $\bdry T$.   
See Figure~\ref{fig:Alexander_braiding}.

Finally, compactifying to $S^3$, the plane $S_\delta$ becomes the desired sphere that contains the twisting circle $c$ and braid axis $A$.
\end{proof}

\begin{figure}[h]
\centering
\includegraphics[width=0.8\textwidth]{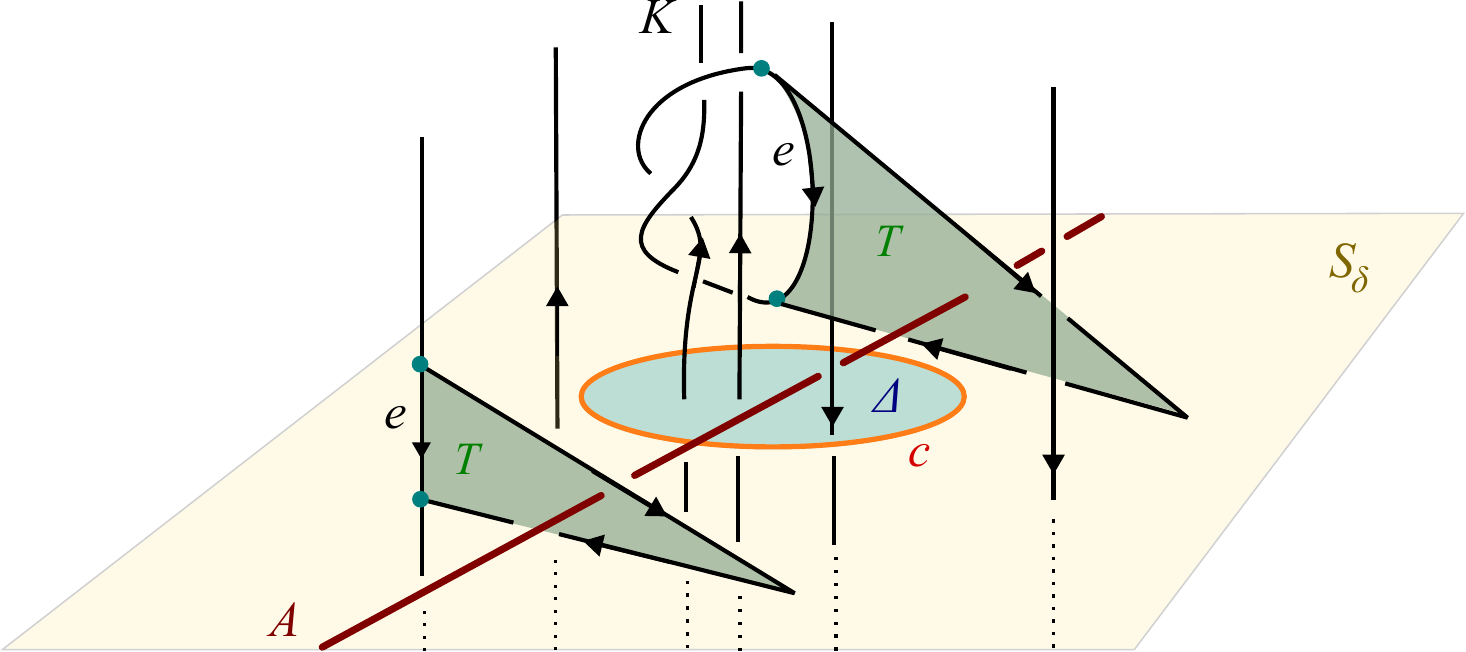}
\caption{Two triangle moves for braiding the oriented knot $K$ about the axis $A$}
\label{fig:Alexander_braiding}
\end{figure}

Start now from the diagram for $K \cup c$ in Figure~\ref{fig:braidtwist}(left) as ensured by Claim~\ref{claim:braidtwist} in which $K$ is braided with the braid diagram $D$.
Then, following the moves of Figure~\ref{fig:braidtwist}, a single twist along $c$ increases the braid index of the shown closed braid presentation of $K$ by $(\eta-\omega)/2$.  

\begin{figure}
\centering
\includegraphics[width=.8\textwidth]{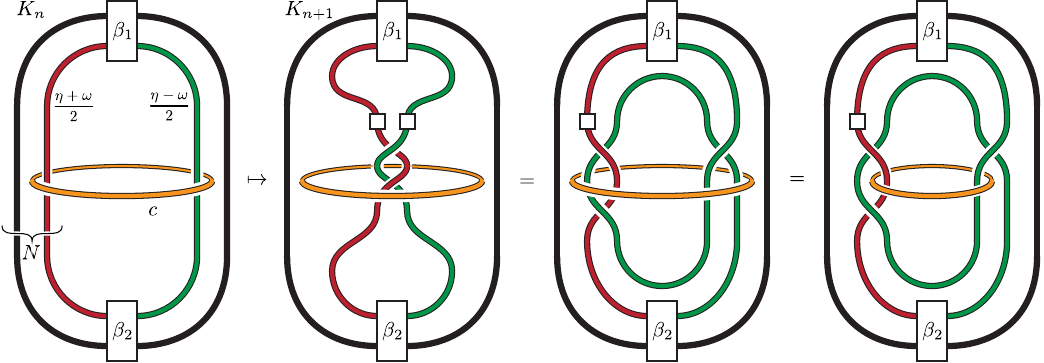}
\caption{A knot $K_n$ with a twisting circle $c$ has a diagram as in the first picture. 
Parallel strands are grouped together. 
A small square indicates a full positive twist on the group of strands running through it. 
Here $K_n$ is shown as a closed braid with braid index $N$; both $\beta_1$ and $\beta_2$ represent $N$--braids. 
The second, third, and fourth show the knot $K_{n+1}$ resulting from a full twist along $c$ followed by some isotopies into a closed braid with braid index $N+(\eta-\omega)/2$.
}
\label{fig:braidtwist}
\end{figure}

Hence $b(K_n) \le b(D) + n(\eta-\omega)/2$, hence 
\begin{equation*}
\overline{b_s}(K_n) = \varlimsup_{n \to \infty}\frac{b(K_n)}{n} \le 
\lim_{n \to \infty}\frac{b(D) + n\eta(\eta-1)/2}{n} 
=  (\eta-\omega)/2
\end{equation*}
as claimed.
\end{proof}

\bigskip

\section{Stable invariants of knots in a twist family with wrapping number 2 and winding number 0.}
\label{sec:wrap2wind0}

In this section we restrict attention to twist families of knots $\{K_n\}$ with twisting circle $c$, wrapping number $\eta=2$, and winding number $\omega=0$.  In particular, $c$ is a `crossing circle' for $K$, bounding a disk that $K$ intersects in $2$ points with opposite sign.  Let $K_{\infty}$ be the `resolution' of this crossing, the link that results from banding $K$ to itself along a simple arc in this disk.

A condition on the degrees of the HOMFLYPT polynomials of $K=K_0$ and $K_{\infty}$ enables the stable braid index to be determined from the Morton-Franks-Williams inequality \cite{Morton,FW}.  Ohyama's bound on crossing number in terms of braid index \cite{Ohyama} then applies to give the stable crossing number.
Here $E(P_L)$ is the maximum degree of $\ell$ in the HOMFLYPT polynomial $P_L(\ell,m)$ of an oriented link $L$.  

\begin{theorem}
\label{thm:eta=2}
Let $\{K_n\}$ be a twist family of knots with $\eta = 2$ and $\omega = 0$. Let $K_{\infty}$ be the associated `resolution' link.   Assume that $E[P_{K_0}] \ne -1+E[P_{K_{\infty}}]$.
Then $b_s(K_n) = 1$ and $c_s(K_n) = 2$. 
\end{theorem}

Similar work on the braid index of twisted satellite links of wrapping number $2$ was done by Nutt \cite{Nutt}. Indeed it follows from the paragraph after the proof of \cite[Lemma 3.11]{Nutt} that our assumption on the degrees of the HOMFLYPT polynomials may be replaced by the assumption that $P_{K_0 \sqcup U} \neq P_{K_\infty}$.

\medskip

Let us now  recall basic facts about the HOMFLYPT polynomial; see \cite{Lickorish} for example.
The HOMFLYPT polynomial $P_L(\ell, m)$ of an oriented link $L$ is defined using skein relations: 
\begin{itemize}
\item $P(\textrm{unknot}) = 1$
\item $\ell P_{L_+} + \ell^{-1} P_{L_-} + m P_{L_0} = 0$, 
\end{itemize}
where $L_{+}, L_{-}, L_0$ are links formed by crossing smoothing changes on a local region of a link diagram, 
as indicated in Figure~\ref{fig:skein}. 

\begin{figure}[h]
\includegraphics[width=.4\textwidth]{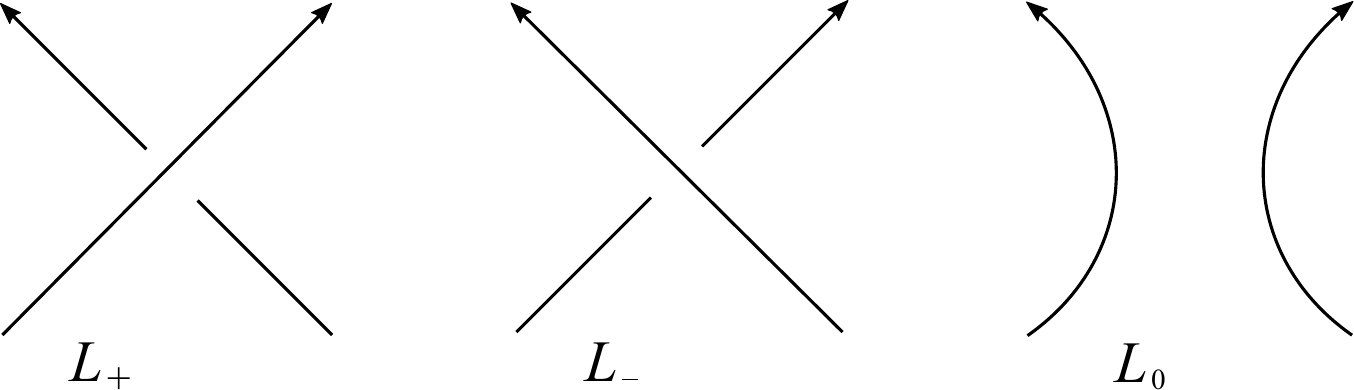}
\caption{Skein triple}
\label{fig:skein}
\end{figure}

The HOMFLYPT polynomial $P_L(\ell, m)$ specializes to the Jones polynomial $V_L(A)$ 
as follows: 
\[
V_L(A) = P_L(i A^4, i(A^{-2}-A^2)), 
\]
where $i = \sqrt{-1}$. 
Furthermore, it is an elementary exercise to show that $V_L(\sqrt{-1}) = (-2)^{r-1}$ when $L$ is a link of $r$ components.
Thus $V_L(A)  \ne 0$ for any link $L$.  Hence this specialization shows that $P_L(\ell, m)$ cannot be zero for any link as well.

\begin{proof}
First we determine the HOMFLYPT polynomial of $K_n$ in terms of the HOMFLYPT polynomial of $K=K_0$ and of the link $K_{\infty}$.

For $n\geq1$, the skein relation at a crossing in the twist region gives the links $L_+ = K_{n-1}$, $L_- = K_{n}$, and $L_0 = K_{\infty}$ since $\omega=0$.
Therefore we have
\[ P_{K_{n}} =- \ell^{2} P_{K_{n-1}} - \ell m P_{K_{\infty}}.  \]
So inductively we may then calculate:
\begin{align*}
P_{K_n} &=- \ell^{2} (- \ell^{2} P_{K_{n-2}} - \ell m P_{K_{\infty}}) - \ell m P_{K_{\infty}}  \\
&=(- \ell^{2})^2 P_{K_{n-2}} +((- \ell^{2}) + 1)(- \ell m) P_{K_{\infty}} \\
&=(- \ell^{2})^2 (- \ell^{2} P_{K_{n-3}} - \ell m P_{K_{\infty}}) +((- \ell^{2}) + 1)(- \ell m) P_{K_{\infty}} \\
&=(- \ell^{2})^3 P_{K_{n- 3}} +((- \ell^{2})^2 +(- \ell^{2}) + 1)(- \ell m) P_{K_{\infty}} \\
&= (- \ell^{2})^{n} P_{K_{0}} +((- \ell^{2})^{n-1} + \dots + (- \ell^{2})^2 +(- \ell^{2}) + 1)(- \ell m) P_{K_{\infty}}
\end{align*}

Since $P_{K_{\infty}}(\ell, m) \ne 0$ and $E[P_{K_0}] \ne -1+E[P_{K_{\infty}}]$ by assumption, we have
\begin{align*}
E[P_{K_n}] &= \max( 2n+E[P_{K_0}], 2n-1+E[P_{K_{\infty}}]) = 2n+ \max( E[P_{K_0}], -1+E[P_{K_{\infty}}])\\
e[P_{K_n}]&= \min( 2n+e[P_{K_0}], 1+e[P_{K_{\infty}}]).
\end{align*}
So for $n \gg 0$, $E[P_{K_n}] =  2n+E_0$ for some constant $E_0$ and 
$e[P_{K_n}]=  1+e[P_{K_{\infty}}]$.
The Morton-Franks-Williams inequality \cite{Morton,FW} 
\[
b(K_n) \ge \frac{1}{2}(E[P_{K_n}] - e[P_{K_n}]) + 1
\]
gives 
$b(K_n) \geq \frac12 ((2n +E_0) - (1+e[P_{K_\infty}]) + 1 = n + \frac{1}{2}(E_0 -e[P_{K_{\infty}}] + 1)$
from which we obtain $\underline{b_s}(K_n)\geq 1$.  Since we know $\overline{b_s}(K_n) \leq \frac{2-0}{2}=1$ by Theorem~\ref{thm:stablebraid}, it follows that $b_s(K_n) = 1$.

Since $2b(K)-2 \leq c(K)$  by \cite{Ohyama}, we also have $2\underline{b_s}(K_n) \leq \underline{c_s}(K_n)$.   Because we have shown that $\underline{b_s}(K_n)=b_s(K_n)\geq 1$, it follows that $\underline{c_s}(K_n)\geq 2$.  
However, because $\eta=2$ so that Theorem~\ref{thm:stablecrossing} implies $\overline{c_s}(K_n) \leq \eta(\eta-1) =2$, we may conclude that $c_s(K_n)=2$.
\end{proof}

\section{Examples}
\label{sec:examples}

We conclude with a couple of examples. 

\begin{example}[{\bf twisted cables}]
\label{cables}
Consider the $(p,q)$--cable $K_{p,q}$ of a knot $K$ where $\eta=q >1$ is the wrapping number.  
Then $K_{p, q}$ lies in interior of a tubular neighborhood $V$ of $K$ with meridian $c$ that $K_{p,q}$ coherently wraps around $q$ times. 
Twisting $K_{p,q}$ about this meridian $c$, 
we obtain a twist family of cable knots $\{ K_{p+qn, q} \}$. 
Then $c_s(K_{p+qn, q}) = q(q-1)$ by Theorem~\ref{thm:stablecrossing} independent of the companion knot $K$. 
\end{example}

\begin{figure}[!ht]
\includegraphics[width=0.5\linewidth]{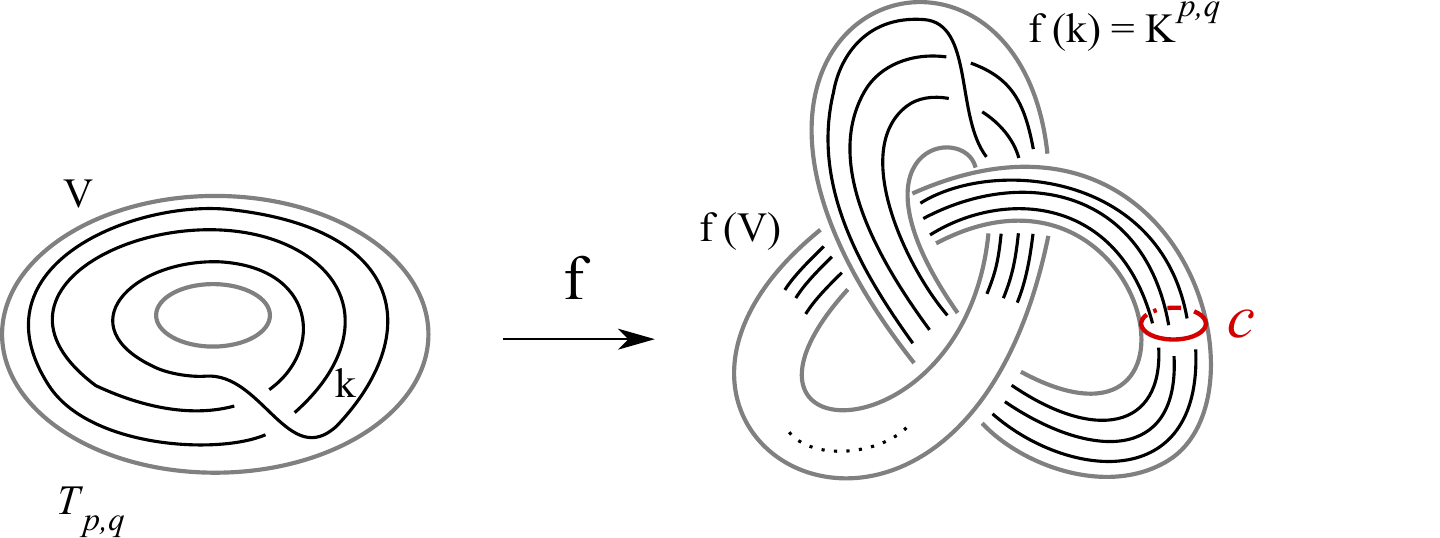}
\caption{A $(p, q)$--cable knot $K_{p, q}$ and twisting circle $c$; $q= 3$}
\label{cable}
\end{figure}

\medskip
\begin{example}
\label{alternating}
Let $K$ be a knot with a diagram $D$ and two unknots $c$ and $c'$ encircling a crossing as in Figure~\ref{wrap2_twist_2} so that $c$ is coherent and $c'$ is not. 
Hence Theorems~\ref{thm:stablecrossing} and \ref{thm:stablebraid} apply to $K \cup c$, but we must check some HOMFLYPT polynomials to apply Theorem~\ref{thm:eta=2} to $K \cup c'$. 
Since the knot $K$ is an unknot, $P_K(\ell,m)=1$. For the twisting circle $c'$, the link $K_{\infty}$ is the figure eight knot with a meridional circle.  
Hence $\eta=2$ since otherwise $K_{\infty}$ would be a split link. 
\[P_{K_{\infty}}(\ell,m) = -{\ell m^{3}} + ({2}{\ell^{3}} + {2}{\ell} + {\ell^{-1}}){m} + (-{\ell^{5}} - {2}{\ell^{3}} - {2}{\ell} - {\ell^{-1}}){m^{-1}}\]
as calculated by KLO \cite{KLO}.  
Therefore $E[P_{K_{\infty}}] - 1 = 5-1=4$ is not $E[P_{K}] = 0$ so that Theorem~\ref{thm:eta=2} does indeed apply to $K \cup c'$.
Hence for the twisting circles $c$ and $c'$ we have
\[
b_s(K_{c,n}) = 0, \quad c_s(K_{c,n})=2 \qquad \mbox{and} \qquad b_s(K_{c',n}) = 1, \quad c_s(K_{c',n}) = 2.
\]

\begin{figure}[!ht]
\includegraphics[width=0.8\linewidth]{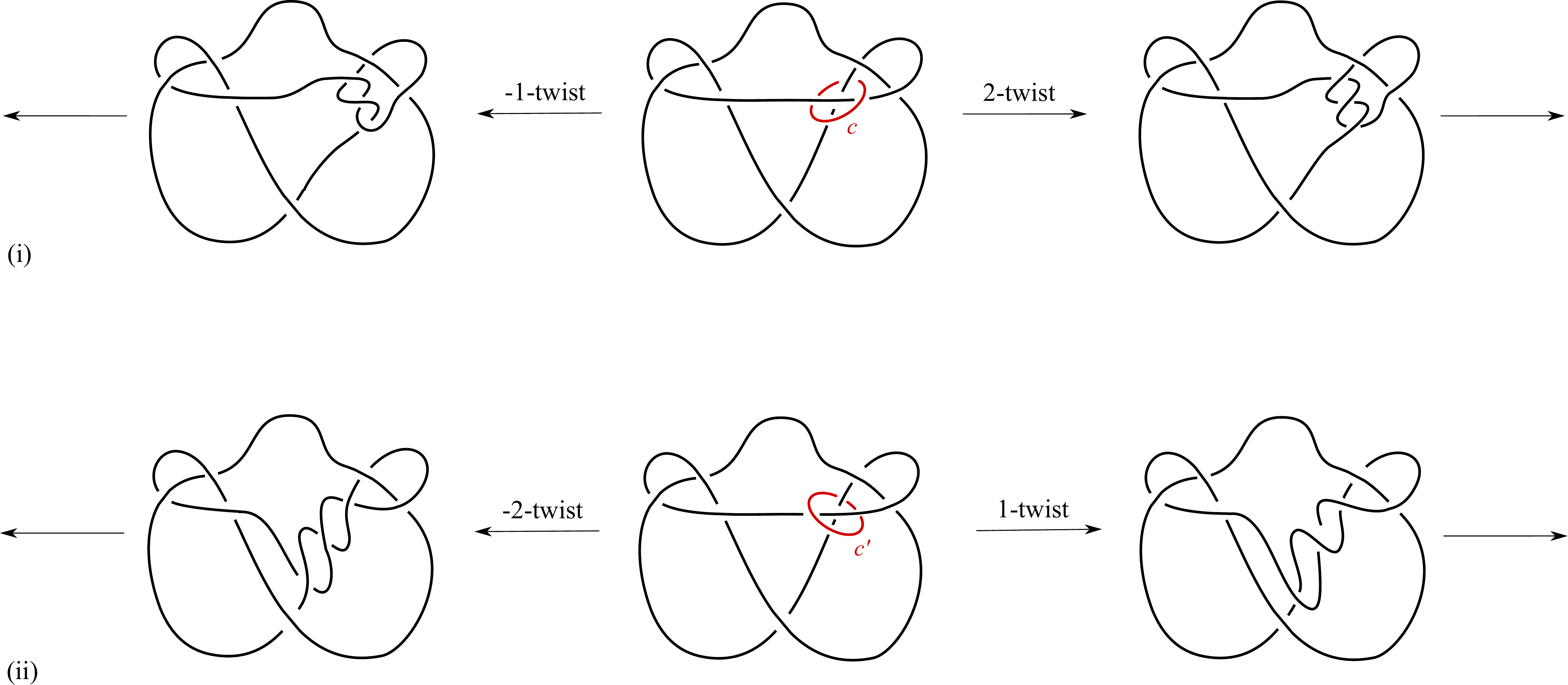}
\caption{Two twist families $\{ K_{c, n} \}$ and $\{ K_{c', n} \}$; The knot $K = K_{c, 0} = K_{c', 0}$ is an unknot.} 
\label{wrap2_twist_2}
\end{figure}
\end{example}

\bigskip
\textbf{Acknowledgments}\quad 

KLB has been partially supported by the Simons Foundation gift \#962034.
He also thanks the University of Pisa for their hospitality where part of this work was done.

KM has been partially supported by JSPS KAKENHI Grant Number JP19K03502, 21H04428 and Joint Research Grant of Institute of Natural Sciences at Nihon University for 2023.

\end{document}